\documentclass{amsart}
\usepackage[all]{xy}
\usepackage{amsmath}
\usepackage{amssymb}
\usepackage{amsthm}
\usepackage{amscd}
\usepackage{enumerate}

\newcommand{\N}{\mathcal{N}}

\newcommand{\F}{\mathcal{F}}

\newcommand{\s}{\Bbb S}
\newcommand{\ppt}{\Pi_t}

\newcommand{\Z}{\Bbb Z}

\newcommand{\RR}{\Bbb R}

\newcommand{\C}{\Bbb C}
\newcommand{\ee}{\left}
\newcommand{\rr}{\right}

\newtheorem{theorem}{Theorem}

\newtheorem{proposition}[theorem]{Proposition}
\newtheorem*{prop}{Proposition \ref{tech}}
\newtheorem{cor}[theorem]{Corollary}
\newtheorem{lemma}[theorem]{Lemma}

\newtheorem{remark}[theorem]{Remark}

\newcommand{\tn}{{\tilde N}}

\numberwithin{theorem}{section}

\begin{document}
\title{Dimension-Free $L^p$-Maximal Inequalities in $\Z_{m+1}^N$}
\maketitle

\begin{abstract}
For $m \geq 2$, let $(\Z_{m+1}^N, |\cdot|)$ denote the group equipped with the $l^0$ (aka Hamming) metric,
\[ |y|= \left| \big( y(1), \dots, y(N) \big) \right| := \left|\{ 1 \leq i \leq N : y(i) \neq 0 \}\right|,\]
and define the $L^1$-normalized indicator of the $r$-sphere,
\[ \sigma_r := \frac{1}{\ee|\{|x| = r\}\rr|} 1_{\{ |x| = r\}}.\]
We study the $L^p \to L^p$ mapping properties of the maximal operator
\[ M^N f (x) := \sup_{r \leq N} | \sigma_r*f| \]
acting on functions defined on $\Z_{m+1}^N$.

Specifically, we prove that for all $p>1$, there exist absolute constants $C_{m,p}$ so that
\[ \ee\| M^N f \rr\|_{L^p(\Z_{m+1}^N)} \leq C_{m,p} \| f \|_{L^p(\Z_{m+1}^N)} \]
for all $N$.

This result may be viewed as an extension of the main theorem of \cite{BK} -- the existence of dimension-free $L^p$-bounds for $p > 1$ for the spherical maximal function in the hypercube, $\Z_{2}^N$.
Indeed, our approach is that of \cite{BK}, which grew out of the arguments of \cite{HKS}, which were in turn motivated by the spectral technique developed in \cite{NS} and \cite{S} in the context of pointwise ergodic theorems on general groups.
\end{abstract}

\section{Introduction}\label{intro}
In $\RR^N$, let
\[ M_{B}^{\RR^N}f(x) := \sup_{r > 0} \frac{c_N}{r^N} \int_{|y| \leq r} |f(x-y)| \ dy,\]
denote the standard Hardy-Littlewood maximal function, where $c_N^{-1}$ is the volume of the $N$-dimensional Euclidean unit ball.

A celebrated result of Stein and Str\"{o}mberg \cite{SS} in Euclidean harmonic analysis concerns the following dimension-independent bounds:
\begin{theorem}[Theorem A of \cite{SS}]
For each $p > 1$ there exists a constant $A_p$ independent of $N$ so that
\[ \left\| M_{B}^{\RR^N}f \right\|_{L^p(\RR^N)} \leq A_p \| f \|_{L^p(\RR^N)}.\]
In particular, while the maximal operators are themselves dimension-dependent, they are all uniformly bounded in $L^p \to L^p$ operator-norm by the same constant, $A_p$.
\end{theorem}

This result was more recently extended by Bourgain \cite{B} to the \emph{cubic} maximal function
\[ M_Q^{\RR^N}f(x) := \sup_{r > 0} \frac{1}{(2r)^N} \int_{y \in Q_r} |f(x-y)| \ dy,\]
where
\[ Q_r := \ee\{ y = (y(1),\dots, y(N)) : |y(i)|\leq r \ \text{ for each } 1 \leq i \leq N \rr\} \]
is the cube of side-length $2r$ centered at the origin.

\begin{theorem}[Theorem of \cite{B}]
For each $p > 1$ there exist constants $A_p'$ independent of $N$ so that
\[ \left\| M_Q^{\RR^N}f \right\|_{L^p(\RR^N)} \leq A_p' \| f \|_{L^p(\RR^N)}.\]
\end{theorem}

The purpose of this article is to establish comparable dimension independent bounds in a discrete setting.

Specifically, for $m \geq 2$, let $\Z_{m+1}^N$ denote the group equipped with the so-called $l^0$-metric,
\[ |y|:= \ee|\{ 1 \leq i \leq N : y(i) \neq 0 \}\rr|.\]
We also define the ($L^2$-normalized) characters
\[ \chi_S(x):= \frac{1}{(m+1)^{N/2}} \xi_m^{x \cdot S} = \frac{1}{(m+1)^{N/2}} \prod_{i \in S} (\xi_m)^{s(i) x(i)} \]
where $\xi_m := e^{2\pi i /(m+1)}$ is a primitive $(m+1)$th root of unity.
Define the the Fourier transform
\[ \F f(S) = \hat{f}(S) = \sum_{x \in Z_{m+1}^N} f(x) \chi_S(x),\]
and the $L^1$-normalized indicator function of the $r$-sphere:
\[
\sigma_r := \frac{1}{|\s_r|} 1_{\s_r}.
\]
We adopt the convention that $\sigma_r$ is $0$ and the respective spheres are empty for $r<0$.

Motivated by \cite{B} and \cite{SS}, we will be interested in establishing dimension-independent bounds for the family of maximal functions
\[ Mf(x) = M^N f(x):= \sup_{r \leq N} |\sigma_r*f| \]
acting on functions in $\Z_{m+1}^N$.

We establish the following theorem:
\begin{theorem}\label{MAIN}
For any $p > 1$, and any $m \geq 2$, there exists a constant $C_{p,m}$ so that
\[ \left\| M f \right\|_{L^p(\Z_{m+1}^N)} \leq C_{p,m} \|f\|_{L^p(\Z_{m+1}^N)}.\]
In particular, the above bounds exist independent of the dimension, $N$.
\end{theorem}

A similar problem was studied in the $m=1$ case in \cite{HKS}:
\begin{theorem}[\cite{HKS}, Theorem 1]
There exists a constant $C_2$ so that for all $N$,
\[ \left\| \sup_{r \leq N} | \sigma_r *f|  \right\|_{L^2(\Z_2^N)} \leq C_2 \|f\|_{L^2(\Z_2^N)}.\]
\end{theorem}
and later in \cite{BK}:
\begin{theorem}[\cite{BK}, Theorem 2.2]\label{big}
For any $p > 1$, there exist constants $C_p$ so that that all $N$,
\[ \left\| \sup_{r \leq N} | \sigma_r *f| \right\|_{L^p(\Z_2^N)} \leq C_p \| f \|_{L^p(\Z_2^N)}. \]
\end{theorem}

Note that, because spherical maximal functions pointwise dominate ball maximal functions, Theorems \ref{big} and \ref{MAIN} also establish dimension independent bounds on the ball average maximal operators for $\Z_{m+1}^N$ for all $m \geq 1$.

\begin{remark}
Theorems \ref{MAIN} and \ref{big} can be viewed as statements about Cartesian powers of finite cliques. The Hamming metric on $\Z_{m+1}^N$  can be isometrically identified with the graph distance metric $K_{m+1}^N$, the Cartesian power of the clique on $m+1$ vertices. To see this one simply labels the vertices of $K_{m+1}^N$ by the elements of $\Z_{m+1}^N$ and computes the distances directly.

As such, our results can be equivalently stated as dimension-independent $L^p$ bounds for Cartesian powers of finite cliques for all $p>1$. Each proof below can be readily rephrased in graph theoretic language: Translation is replaced by composition with a graph automorphism, the Fourier transform is replaced by a change of basis to diagonalize the adjacency matrix, Fourier coefficients are eigenvalues spherical averaging matrices, etc.
\end{remark}

The arguments of \cite{HKS} and \cite{BK} are applications of Stein's method \cite{S}, used in extending the well-known Hopf-Dunford-Schwartz maximal theorem for semigroups to more ``singular'' maximal averages, and breaks into four main steps:

\begin{enumerate}
\item By comparison with the noise semigroup from Boolean Analysis \cite[\S 4]{HKS}, \cite[\S 3]{BK} the ``smoother'' maximal function
\[ \sup_{K } \frac{1}{K+1} \bigg| \sum_{k\leq K} \sigma_k*f \bigg|, \]
is shown to satisfy a dimension-free weak-type $1-1$ inequality:
\[ \left| \left\{ \sup_{K } \frac{1}{K+1} \bigg| \sum_{k\leq K} \sigma_k*f\bigg| > \lambda \right\}\right| \leq \frac{C_1 \|f\|_{L^1(\Z_2^N)}}{\lambda}, \ \lambda \geq 0; \]
\item The ``rougher'' maximal function $\sup_{r \leq N} | \sigma_r *f|$ is compared to the ``smoother'' maximal function in $L^2$ by using Littlewood-Paley theory on the group $\Z_{m+1}^N$. The key tool is an analysis of the (radial) spherical multipliers
\[ \F \sigma_k(S) := \kappa_{k}^N(S)\]
the \emph{Krawtchouk} polynomials, which are introduced and discussed in \cite[\S 2]{HKS};
\item The ``rough'' maximal function, $\sup_{r \leq N} | \sigma_r *f|$, is compared to increasingly ``rougher'' maximal functions in $L^2$. Analysis of the Krawtchouk polynomials are pivotal in these further comparisons;
\item Stein interpolation is used to control $\sup_{r \leq N} | \sigma_r *f|$ in $L^p, \ p > 1$.
\end{enumerate}

A portion of our approach is based on this methodology. However there are some obstacles that require different techniques, notably bounding averages over spheres of sufficiently large radius. The symmetry of the $N$ dimensional hypercube about the radius $N/2$ allows the maximal averaging overator over distant spheres (i.e. spheres of radius greater than $N/2$) to immediately follow from those over the spheres. Indeed, distant spheres can simply be viewed as local spheres centered at an antipodal point. The loss of symmetry requires an additional argument to bound the maximal averaging operator over the analogous distant spheres.

Theorems \ref{MAIN} and \ref{big} together synthesize a generalization to arbitrary direct sums of finite cyclic groups, which can be viewed as a statement about all finite abelian groups.

Let $n, N_1, \dots , N_n \in \mathbb N$ and let $A_m$ be the group
\[
\Z_{m + 1}^{N_m}
\]
for $m \geq 1$. Also let $A$ be the direct sum
\[
A := \oplus_{m=1}^n A_m
\]
with the notation
\[
y = \oplus_{m=1}^n \ee(y_m^1, \dots , y_m^{N_m}\rr)
\]
equipt with the modified $l^0$ metric,
\[
|y| := \left|\{(m,j): 1 \leq m \leq n, 1 \leq j \leq N_m, y_m^j \neq 0\}\right|.
\]
Put more simply, viewing $y$ as a $N_1 + \dots + N_n$-tuple in the natural way, $|y|$ is the number of nonzero components. Let $\sigma^A_r$ be the $L^1$-normalized indicator function of the radius $r$, i.e.
\[
\sigma^A_r(x) := \frac{1}{\ee|\{|x| = r\}\rr|} 1_{\{|x| = r\}}
\]
and define the operator
\[
M^A f(x) := \sup_{r > 0} \ee|\sigma^A_r * f(x)\rr|,
\]
the spherical maximal operator.

\begin{theorem}\label{directsumbound}
For any $p > 1$ there exist constants $C_{p,n}$ such that $\|M^A f\|_{L^p(A)} \leq C_{p,n} \|f\|_{L^p(A)}$. In particular $C_{p,n}$ has no direct dependence on $A$.
\end{theorem}

\begin{proof}
Notice that the $L^1$-normalized indicator of a sphere in $A$ is a convex combination of products of $L^1$-normalized indicators of spheres in $A_1, \dots , A_n$. Thus the spherical maximal function on $A$ is pointwise dominated by the product of the spherical maximal functions on $A_1, \dots , A_n$. By Theorems \ref{big} and \ref{MAIN}, for a fixed $m$, the spherical maximal operator on $A_m$ satisfy $L^p$ inequalities dependent only on $p$ and $m$. Thus the product of the spherical maximal functions on $A_1, \dots , A_n$ depends only on $p$ and $n$.
\end{proof}

This result admits a corollary concerning Cayley graphs of finite abelian groups.

\begin{cor}\label{cayleycorollary}
Let $A$ be a finite abelian group whose elements have order at most $d$. Then there exists a generating set $S$ of minimal size up to a factor of $d$ such that the spherical maximal operator on the Cayley graph $\Gamma(A,S)$ satisfies $L^p$ bounds for all $p > 1$ dependent only on $p$ and $n$.
\end{cor}

\begin{proof}[Proof of Corollary \ref{cayleycorollary}, Assuming Theorem \ref{directsumbound}]
If $A$ is a finite abelian group that admits a minimal size generating set with $s$ elements, by the fundamental theorem of finitely generated abelian groups there exist $m_1 < \dots < m_k < \infty $ and $\tn_1 + \dots + \tn_k = s$ such that we can identify $A$ with
\begin{align}\label{directsumisomorphism}
\oplus_{j=1}^k \Z_{m_j + 1}^{\tn_j}.
\end{align}
We examine the generator set $S$ of $s$-tuples that have exactly one nonzero component. Note that as long as each element of $A$ has order at most $d$, $|S| \leq sd$ so $S$ is a generating set of minimal size up to a factor of $d$. Setting $n := m_k$, we can identify \eqref{directsumisomorphism} with
\[
\oplus_{m=1}^n \Z_{m + 1}^{N_m}
\]
where in general $N_m = 0$ for some values of $m$. Note that the distance metric on the Cayley graph $\Gamma(A,S)$  is precisely the $l^0$ metric of Theorem \ref{directsumbound}. From here the corollary is a direct application of Theorem \ref{directsumbound}.
\end{proof}

\bigskip

The structure of the paper is as follows:

In $\S \ref{noise}$, we introduce our smooth spherical maximal operator for local (i.e. small radius) spheres, and prove that they satisfy dimension independent weak type $1-1$ bounds;

In $\S \ref{comparison}$, we review Stein's semigroup comparison method, and adapt it to our present context; assuming the technical Proposition \ref{tech}, we prove $L^p$ bounds for the local spherical maximal operator;

In $\S \ref{distant}$, we retool the arguments of $\S \ref{noise}$ and $\S \ref{comparison}$ and use them to bound the distant maximal operator, thereby establishing main result, Theorem \ref{MAIN}, modulo the proof of Proposition \ref{tech}; and

In $\S \ref{krawtchouk}$, we prove Proposition \ref{tech}.

\subsection{Notation}
Throughout, we denote $c_m:= \frac{m}{m+1}$. When clear from context, we will suppress the superscript ``$N$'' in the definition of our maximal functions. We will also make use of the modified Vinogradov notation. We use $X \lesssim Y$, or $Y \gtrsim X$ to denote the estimate $X \leq CY$ for some constant $C$ which may depend only on $m$ (in general we will suppress dependence on $m$). If we need $C$ to depend on a parameter other than $m$, we shall indicate this by subscripts, for instance $X \lesssim_a Y$ denotes the estimate $X \leq C_a Y$ for some $C_a$ depending on $a$. We use $X \approx Y$ as shorthand for $X \lesssim Y \lesssim X$, and similarly for $X \approx_a Y$.

\section{The Smooth Local Spherical Maximal Operator in $\Z_{m+1}^N$}\label{noise}

As alluded to in the introduction, the lack of symmetry of $\Z_{m+1}^N$ requires separate treatment of local and distant spheres, so we split $M$ in two separate maximal operators:
\[ M f(x) \leq M^Lf + M^Df:= \sup_{k \leq c_mN} |\sigma_k*f| + \sup_{k> c_mN} |\sigma_k*f|.\]

In this section, we prove:
\begin{proposition}\label{smoothweaktype}
The smooth local spherical maximal function
\[ M^L_Sf:=\sup_{L \leq  c_m N } \bigg|\frac{1}{K+1} \sum_{k \leq K} \sigma_k *f\bigg| \]
is of weak-type $1-1$, with bound independent of $N$,
i.e. there exists some absolute $C_{1,m}$ so that
\[ \|M^L_Sf\|_{L^{1,\infty}(\Z_{m+1}^N)} \leq C_{1,m} \|f\|_{L^1(\Z_{m+1}^N)}.\]
\end{proposition}

Proposition \ref{smoothweaktype} will also be useful in $\S \ref{distant}$ to bound the smooth distant spherical maximal operator given by
\[
M^D_S f := \sup_{D \leq \frac{N}{m+1}} \bigg|\frac{1}{D+1} \sum_{k \leq K} \sigma_{N-k} *f\bigg|.
\]
Following the lead of \cite[\S 4]{HKS}, we bound $M_S^L$ by comparison with an appropriate ``noise'' semigroup, which we now introduce.

\subsection{The noise semigroup in $\Z_{m+1}^N$}
For fixed $0 < p < c_m$ we define a probability measure $\tilde\mu_p$ on $\Z_{m+1}$ given by

\[ \tilde\mu_p\ee(\{w\}\rr) := \begin{cases} 1-p &\mbox{if } w = 0 \\
\frac{p}{m} & \mbox{otherwise}, \end{cases} \]
and for $y \in \Z_{m+1}^N$,
\[ \tilde\mu^N_p\ee(\{y\}\rr) = \left( \frac{p}{m} \right)^{|y|} \left( 1- p \right)^{N-|y|},\]
where, as above,
\[ |y|:= \left|\{ 1 \leq i \leq d : y(i) \neq 0 \}\right|\]
is the $l^0$-metric. We view $\tilde\mu^N_p$ alternatively as a measure and a function depending on context.

Consider the (dimension dependent) convolution operator
\[ \tilde{\N}_p f(x):= f* \tilde\mu_p^N(x) = \int_{\Z_{m+1}^N} f(x+y) \tilde\mu_p^N(y).\]
We denote by $\xi = \xi_m$ a primitive $(m+1)$th root of unity.
\begin{lemma}\label{characternoise}
For each ($L^2$-normalized) character
\[
\chi_S(x) := \frac{1}{(m+1)^{N/2}} \xi^{S \cdot x} = \frac{1}{(m+1)^{N/2}} \prod_{i=1}^N \xi^{S(i) x(i)}
\]
where $S,x \in \mathbb Z_{m+1}^N$ and $S \cdot x = \sum_{i=1}^N S(i) x(i)$, we have
\[ \tilde{\N}_p \chi_S(x) = (1 - p/c_m )^{|S|} \chi_S(x).\]
\end{lemma}

\begin{proof}
First note that
\[
\chi_S(x+y) = \frac{1}{(m+1)^{N/2}} \xi^{(x+y) \cdot S} = \chi_S(x) \xi^{y \cdot S}
\]
Thus
\begin{align}\label{coordinateform}
\tilde{\mathcal N}_p \chi_S(x) = \chi_S(x) \int_{\mathbb Z_{m+1}^N} \xi^{y \cdot S} \tilde\mu_p^N(y) = \chi_S(x) \int_{\mathbb Z_{m+1}^N} \prod_{i=1}^N \xi^{y(i) S(i)} \tilde\mu_p^N(y)
\end{align}
However, $\tilde\mu_p^N$ is a Cartesian product of $N$ copies of $\tilde\mu_p$ so \eqref{coordinateform} can be written
\begin{align}\label{productform}
\tilde{\mathcal N}_p \chi_S(x) = \chi_S(x) \prod_{i=1}^N \int_{\mathbb Z_{m+1}} \xi^{y S(i)} \tilde\mu_p(y)
\end{align}

If $S(i) = 0$, the integral in \eqref{productform} evaluates to $1$ because the integrand is $1$ and $\tilde\mu_p$ is a probability measure.

If $S(i) \neq 0$, 
\begin{align*}
\int_{\mathbb Z_{m+1}} \xi^{y S(i)} \tilde\mu_p(y) &= \tilde\mu_p(\{0\}) + \sum_{y=1}^m (\xi^{S(i)})^y \tilde\mu_p(\{y\})\\
&= (1-p) + \frac{p}{m} \sum_{y=1}^m (\xi^{S(i)})^y\\
&= 1 - p - \frac{p}{m}\\
&= 1 - p/c_m
\end{align*}

Splitting the factors in \eqref{productform} into those corresponding to $0$ and non-$0$ indices of $S$, we see
\[
\tilde{\mathcal N}_p \chi_S(x) = \chi_S(x) \bigg[\prod_{i: S(i) \neq 0} \left(1 - p/c_m\right)\bigg] = \chi_S(x) (1 - p/c_m)^{|S|}
\]
\end{proof}

Consequently, with $p(t) = c_m(1- e^{-t})$ and
\[ \aligned
\mu^N_t(y) &:= \tilde\mu^N_{p(t)}(y)\\
\N_t &:= \tilde{\N}_{p(t)}
\endaligned \]
(so $\tilde{\N}_p = \N_{-\ln(1- p/c_m)}$), we have
\[ \N_t \chi_S(x)= e^{-t|S|} \chi_S(x),\]
and thus the family of operators $\{N_t : t > 0\}$ form a semigroup, and the maximal operator $\N_*$ given by
\[ \N_* f:= \sup_{T} \left|\frac{1}{T} \int_0^T \N_t f \ dt\right|\]
is of weak-type $1-1$, independent of dimension (\cite[Lemma VIII.7.6, pp. 690-691]{DS}).

For the sake of comparison with $M^L_S$, it will be convenient to reparametrize the semigroup maximal function in terms of $p$.

\begin{proposition}\label{reparameterization}
The maximal function
\[ \tilde{\N}_{*}f := \sup_{P \leq c_m} \left|\frac{1}{P} \int_0^P \tilde{\N}_p f \ dp\right| \]
is bounded pointwise by $\N_* f$. In particular $\tilde{\N}_{*}$  is of weak-type $1-1$ independent of dimension.
\end{proposition}
\begin{proof}
By direct calculation, one verifies -- analogous to the proof of \cite[Lemma 9]{HKS} -- that the measure
\[
\nu_P := \bigg\{\frac{c_m}{P} T e^{-T} 1_{(0,-\ln[1-P/c_m])} \bigg\} dT \ + \bigg\{\left(\frac{c_m}{P} - 1\right)\left(-\ln \left[1-\frac{P}{c_m}\right]\right)\bigg\} \delta_{-\ln[1-P/c_m]}
\]
has total mass $1$. Moreover, noting that the bracketed expression in \eqref{semigroupreparameterization} below equals $\frac{1}{P}1_{p\leq P}$, further computation reveals that
\begin{align}\label{semigroupreparameterization}
\frac{1}{P} \int_0^P \tilde\mu^N_p \ dp &= \int_0^{c_m} \tilde\mu^N_p \ \left( \frac{1}{c_m - p} \int_{-\ln(1 - p/c_m)}^\infty \frac{1}{T} \ d\nu_P(T) \right) \ dp\\
&= \int_0^\infty \left( \frac{1}{T} \int_0^{c_m(1- e^{-T})} \tilde\mu^N_p \frac{1}{c_m - p} \ dp \right) \ d\nu_P(T)\nonumber\\
&= \int_0^\infty \left( \frac{1}{T} \int_0^T \mu^N_t \ dt \right) \ d\nu_P(T).\nonumber
\end{align}

Because (for fixed $p$ and $t$) the convolution operators $\tilde\N_p$ and $\N_t$ are given by finite sums, they commute with all convergent integrals in $p$ and $t$. This leads directly to a pointwise majorization
\[ \aligned
\left|\frac{1}{P} \int_0^P \tilde{\N}_pf \ dp\right|  &= \left|f * \left[\frac{1}{P} \int_0^P \tilde\mu_p^N\right]\right|\\
&= \left|f * \left[\int_0^\infty \left( \frac{1}{T} \int_0^T \mu^N_t \ dt \right) \ d\nu_P(T)\right]\right|\\
&= \left|\int_0^\infty \left( \frac{1}{T} \int_0^T \N_tf \ dt \right) \ d\nu_P(T)\right| \\
&\leq \int_0^\infty \N_*f \ d\nu_P(T) \\
&= \N_*f, \endaligned\]
from which the result follows.
\end{proof}

Finally, we will compare the smooth maximal function with the reparametrized ``semigroup'' maximal function:
\begin{proposition} For any nonnegative function $f$ we have the pointwise inequality
\[ M^L_S f \lesssim \tilde{\N}_{*}f. \]
In particular, $M^L_S$ is of weak-type $1-1$, independent of dimension.
\end{proposition}

\begin{proof}
We may express
\[ \aligned
\tilde\mu_p^N &= \sum_{l=0}^N (p/m)^l (1-p)^{N-l} 1_{\s_l} \\
&= \sum_{l=0}^N \binom{N}{l} (p)^l (1-p)^{N-l} \frac{1}{m^l \binom{N}{l}} 1_{\s_l} \\
&= \sum_{l=0}^N B(N,p,l) \sigma_l \endaligned\]
where $B(N,p,l) := \binom{N}{l} (p)^l (1-p)^{N-l}$.

By Lemma \ref{bigkchoice} below (similar to \cite{HKS}), for each $L \leq N$ we can choose $P(L) \in (0, c_m]$ that satisfies the favorable pointwise comparison
\[
\frac{1}{L+1} \lesssim \frac{1}{P(L)} \int_0^{P(L)} B(N,p,l) \, dp
\]
for each $l \leq L$. Thus
\begin{align}\label{smoothnoisepwise}
\sum_{l \leq L} \frac{1}{L+1} \sigma_l &\lesssim \sum_{l \leq N} \left( \frac{1}{P(L)} \int_0^{P(L)} B(N,p,l) \ dp \right) \sigma_l\nonumber\\
\frac{1}{L+1} \sum_{l \leq L} \sigma_l &\lesssim \frac{1}{P(L)}\int_0^{P(L)} \tilde\mu_p^N \ dp.
\end{align}
Noting that all terms in \eqref{smoothnoisepwise} are nonnegative, we observe that for any nonnegative function $f$, we have the pointwise comparison
\[ \aligned
\frac{1}{L+1} \sum_{l \leq L} \sigma_l * f &\lesssim \left(\frac{1}{P(L)}\int_0^{P(L)} \tilde\mu_p^N \ dp\right) * f\\
&= \left(\frac{1}{P(L)}\int_0^{P(L)} \tilde\N_p f \ dp\right)\\
&\leq \tilde\N_* f
\endaligned \]

where the first equality above is justified as in Proposition \ref{reparameterization}. Taking a supremum over all $L \leq c_mN$ provides the desired pointwise inequality. To prove the weak-type bound, first observe that because $M_S$ is a supremum over convolution operators with nonnegative kernels, we immediately have the pointwise inequality $M_S |f| \geq M_S f$ for an arbitrary function $f$. Thus for any $\|f\|_{L^1\left(\Z_{m+1}^N\right)} = 1$
\begin{align}\label{absnoiseineq}
\|M_S f\|_{L^{1,\infty}(\Z_{m+1}^N)} &\leq \big\|M_S |f|\big\|_{L^{1,\infty}(\Z_{m+1}^N)}\nonumber\\
&\lesssim \big\|\tilde\N_* |f|\big\|_{L^{1,\infty}(\Z_{m+1}^N)}
\end{align}

Simply because $f$ and $|f|$ share $L^1$ norms (i.e. $1$), \eqref{absnoiseineq} is bounded by the weak $1-1$ operator norm of $\tilde\N_*$. Taking a supremum over all $L^1$ normalized $f$ then proves that $M_S$ inherits the dimension independent weak-type $1-1$ bound from $\tilde\N_*$.
\end{proof}

Applying the Marcinkiewicz interpolation theorem with the trivial $L^\infty$ bound yields the desired $L^p$ bounds.

\begin{cor}\label{smoothlpbounds}
The operator $M^L_S$ satisfies $L^p$ bounds for all $p > 1$ that depend on $p$ and $m$ but are independent of dimension.
\end{cor}

All that remains in the section is to prove the key Lemma \ref{bigkchoice}. Below we present a computational proof based on an application of Stirling's formula. Remark \ref{probabilisticexplanation} conveys a probabilistic intuition for that computation for this calculation, made rigorous by lemmas \ref{distantiid} and \ref{blbound}). To adapt the remark to this situation, one simply replaces $\sigma_N$ with $\sigma_0$.

We put forth both methods for completeness.

\begin{lemma}\label{bigkchoice}
For each $0 \leq l \leq L \leq c_mN$, there exists $P(L) \in (0,c_m]$ (independent of $N$ and $l$) such that
\[
\frac{1}{L+1} \lesssim \frac{1}{P(L)} \int_0^{P(L)} B(N,p,l) \, dp
\]
\end{lemma}

\begin{proof}
We choose the value $P(L)$ as follows:
\[
P(L) =
\begin{cases}\frac{1}{N} \text{ if } L=0\\
\frac{L}{N} \text{ if } 1 \leq L \leq c_m N.
\end{cases}
\]

Because $\frac{P(L)}{L+1} \approx \frac{1}{N}$, it suffices to prove
\begin{align}\label{reducedsmoothinequality}
\frac{1}{N} \lesssim \int_0^{P(L)} B(N,p,l) \ dp
\end{align}
independent of $N$ and $l$. Also note that if $l=0$ we have
\[
\int_0^{P(L)} B(N,p,0) \ dp \geq \int_0^{1/N} (1-p)^N \ dp \gtrsim \frac{1}{N}
\]
so we can assume $1 \leq l$ (and recall $l \leq L \leq c_m N$). We estimate the right side of \eqref{reducedsmoothinequality} from below by
\[
\int_{l/N-\sqrt l/2N}^{l/N} B(N,p,l) \, dp
\]

From there it will suffice to show that for all $p \in [l/N-\sqrt l/2N,l/N]$ the inequality $B(N,p,l) \gtrsim 1/\sqrt l$ holds. To prove this, we first observe that by a direct application of Stirling's formula, $B(N,l/N,l) \gtrsim 1/\sqrt l$. Then we show that $B(N,p,l)$ maintains this bound for all
\[
\frac{l}{N}-\frac{\sqrt l}{2N} \leq p \leq \frac{l}{N}
\]
as follows:

\begin{align*}
\left| \ln \frac{B(N,l/N,l)}{B(N,p,l)} \right| &= \left| \int_p^{l/N} \partial_t \ln B(N,t,l) \ dt \right|\nonumber \\
&= \left| \int_p^{l/N} \frac{l-Nt}{t(1-t)} \ dt \right|\nonumber \\
&\leq \left(\frac{l}{N} - p\right) \left(\max_{t \in [p,l/N]}\frac{1}{t(1-t)}\right) \left(\max_{t \in [p,l/N]}l-Nt\right)\\
&\lesssim \frac{\sqrt l}{N} \frac{N}{l} \sqrt l \\
&= 1
\end{align*}

Exponentiating, it follows that
\[
\frac{B(N,l/N,l)}{B(N,p,l)} \approx 1 \implies B(N,p,l) \gtrsim \frac{1}{\sqrt l}.
\]
\end{proof}

\section{The Comparisons -- Stein's Method}\label{comparison}

The goal of this section is to prove the local spherical bound and set up much of the machinery for the distant spherical, and thus general spherical, bound.

\begin{theorem}\label{localbound}
The local spherical maximal operator $M^L$ given by
\[
M^L f(x) = \sup_{k \leq c_m N} |\sigma_k * f(x)|
\]
satisfies $L^p$ bounds for all $p > 1$ dependent only on $p$ and $m$.
\end{theorem}

As announced, in this section we adapt the Nevo-Stein \cite{NS} spectral machinery to our present context. We prepare to do so in our first subsection:

\subsection{Krawtchouk Preliminaries}
It is helpful to define the convolution operators:
\[ P^k f(x):= f* \sigma_k(x). \]
Their discrete derivatives
\begin{align}\label{differencetk}
\triangle^0 P^k &:= P^k,\nonumber \\
\triangle^1 P^k &:= P^k - P^{k-1},\nonumber \\
&\ \vdots\nonumber \\
\triangle^t P^k &:= \triangle \ee(\triangle^{t-1} P^k\rr) =
\sum_{j \leq t} (-1)^j \binom{t}{j} P^{k-j}, \\
&\ \vdots\nonumber
\end{align}
and their associated (radial) multipliers
\[ \F \ee(\triangle^t P^k\rr) (|S|)\]
will be central to our study.

First, when $|S|=r$ we have \cite[\S 5.3]{CST}
\[ \F P^k(r) = \sum_{j=\max(0,r+k-N)}^{\min(r,k)} (-1)^j \binom{N}{k}^{-1} \binom{r}{j} \binom{N-r}{k-j} m^{-j} =: \kappa_r^N(k),\]
the $k$th (normalized) Krawtchouk polynomial in $\Z_{m+1}^N$. By expanding the binomial coefficients in the expression above, it is easy to see that $\kappa_r^N(k) = \kappa_k^N(r)$ for all $r,k,N$. We adopt the convention that if any of $r$, $k$, or $N$ is negative, then $\kappa_r^N(k) = 0$.

The Krawtchouk polynomials have the following useful difference relation:

\begin{lemma}\label{krawtchoukderivatives}
In $\Z_{m+1}^N$, if $r \geq 0$ and $k \geq 1$ are integers, then
\[
(\F \triangle P^k)(r) = \kappa_k^N(r) - \kappa_{k-1}^N(r) = \kappa_r^N(k) - \kappa_r^N(k-1) = \ee(-\frac{1}{c_m}\rr)\frac{\binom{N-1}{r-1}}{\binom{N}{r}} \kappa^{N-1}_{k-1}(r-1).\]
\end{lemma}

\begin{proof}
For the boundary case $r=0$, direct computation shows that $\F P^k(r) = 1$ for all $k$. Thus, $\F \triangle P^k(r) = \F P^k(r) - \F P^{k-1}(r) = 0$. We may now assume $r$ is positive.

Because dimension is not a constant in the lemma, we adopt the notation
\begin{align*}
&\mathbb S_r^N = \ee\{x \in \mathbb Z_{m+1}^N: |x|=r\rr\}\\
&\sigma_r^N = \frac{1}{|\mathbb S_r^N|} 1_{\mathbb S_r^N}
\end{align*}

Letting $y_j^N = (1,\dots,1,0,\dots,0)$ with $j$ `$1$'s and $N-j$ `$0$'s, we exploit the radiality of $\mathcal F \sigma_j^N$ to see

\begin{align}
\kappa_r^N(k) - \kappa_r^N(k-1) &= \ee\langle \sigma_r^N, \xi^{x \cdot y_k^N} - \xi^{x \cdot y_{k-1}^N}\rr\rangle\nonumber\\
&= \frac{1}{|\mathbb S_r^N|} \sum_{x \in \mathcal S_r^N} (\xi^{x_1 + \dots + x_{k-1}})(\xi^{x_k} - 1)\nonumber\\
&= \frac{1}{|\mathbb S_r^N|} \sum_{x \in \mathcal S_{r-1}^{N-1}} (\xi^{x_1 + \dots + x_{k-1}}) \sum_{x_k=1}^m (\xi^{x_k} - 1)\label{expandeddifference}
\end{align}

The last equality follows from the observation that any summand corresponding to an $x \in \mathbb S_r^N$ such that $x_k = 0$ is $0$. As in the proof of Lemma \ref{characternoise} we have
\[
\sum_{x_k=1}^m (\xi^{x_k} - 1) = -(m+1)
\]
Rearranging \eqref{expandeddifference} then yields
\begin{align*}
-(m+1) \frac{|\mathbb S_{r-1}^{N-1}|}{|\mathbb S_r^N|} \ee\langle \frac{1}{|\mathbb S^{N-1}_{r-1}|} 1_{\mathbb S^{N-1}_{r-1}}, \xi^{x_1 + \dots + x_{k-1}}\rr\rangle &= -(m+1) \frac{|\mathbb S_{r-1}^{N-1}|}{|\mathbb S_r^N|} \ee\langle \sigma_{r-1}^{N-1}, \xi^{x \cdot y_{k-1}^{N-1}}\rr\rangle\\
&= -\frac{{{N-1} \choose {r-1}}}{c_m{{N} \choose {r}}} \kappa^{N-1}_{r-1}(k-1)\\
&= -\frac{{{N-1} \choose {r-1}}}{c_m{{N} \choose {r}}} \kappa^{N-1}_{k-1}(r-1)
\end{align*}
\end{proof}

Applying Lemma \ref{krawtchoukderivatives} $t$ times yields a useful general expression for higher orders differences.

\begin{cor}\label{krawtchoukhigherderivatives}
For any integers $0 \leq t \leq k$ and $0 \leq r$,
\[ (\F \triangle^t P^k)(r)= \ee(-\frac{1}{c_m}\rr)^t \frac{\binom{N-t}{r-t}}{\binom{N}{r}} \kappa^{N-t}_{r-t}(k-t). \]
Notice that, if $r < t$, this means $(\F \triangle^t P^k)(r) = 0$.
\end{cor}

Now we define
\[ \aligned \partial^0 \kappa_r^N(k) &= \kappa_r^N(k), \\
\partial \kappa_r^N(k) &= \partial^1 \kappa_r^N(k) := \kappa_r^N(k) - \kappa_r^N(k-1), \endaligned \]
and $\partial^t \kappa_r^N(k) := \partial ( \partial^{t-1} \kappa_r^N(k))$, provided $t \leq \min\{ r, k\}$. Otherwise we set $\partial^t \kappa_r^N(k) =0$.
Using this notation, for $t \leq k$, we may express
\[ (\F \triangle^t P^k)(r) = \partial_t \kappa_k^N(r) = \ee(-\frac{1}{c_m}\rr)^t \frac{\binom{N-t}{r-t}}{\binom{N}{r}} \kappa^{N-t}_{k-t}(r-t).\]

\begin{remark}
The restriction to $k \geq 1$ in Lemma \ref{krawtchoukderivatives} and (and therefore $k \geq t$ in Corollary \ref{krawtchoukhigherderivatives}) is necessary because the identity $(\F \triangle^t P^k)(r) = \partial_t \kappa_k^N(r)$ breaks down for $t > k$.
\end{remark}

The following proposition, whose proof we defer to \S \ref{krawtchouk} below, is the key quantitative ingredient needed to anchor the argument:
\begin{proposition}\label{tech}
There exists a constant $d$ (dependent only on $m$) such that for all $r,k,N$ we have
\[ |\kappa_k^N(r)| \leq e^{-d \frac{rk}{N}}.\]
\end{proposition}

\subsection{A Review of Nevo-Stein}
In this subsection, we shall regard $N$ as fixed, and
(quickly) review the comparison argument of \cite{S} as it relates to our current setting. For a fuller treatment, we refer the reader to \cite{NS}.

In the last subsection, we introduced the convolution operators $\{P^k\}$.
Because they are self-adjoint, positive $L^1$- and $L^\infty$-contractions, we may use the following outline from \cite{S}, \cite{NS}:

With $\lambda = \alpha + i \beta \in \C$, we recall the complex binomial coefficients
\[ A^\lambda_n = \frac{ (\lambda + 1)(\lambda + 2) \dots (\lambda + n)}{n!}, \ A_0^\lambda:= 1, A_{-n}^\lambda:=0.\]
We define the Ces{\`a}ro sums
\[ S_n^\lambda f(x) := \sum_{k\leq n} A_{n-k}^\lambda P^k f(x), \ \lambda \in \C,\]
for $n \leq c_m N$
and remark that for any integer $0 \leq t \leq k$, we have
\[
S_k^{-t-1}f(x) = \Delta^t P^k f(x)
\]
by a simple computation using \eqref{differencetk}. In particular, because we are only working with $S^\lambda_n$ for $n \leq c_m N$, Corollary \ref{krawtchoukhigherderivatives} shows that whenever $t > c_m N$ we have $S_n^{-t-1} f \equiv 0$.

The maximal functions associated to these higher Ces{\`a}ro means are
\[ S_*^\lambda f(x):= \max_{0 \leq n \leq c_m N} \left| \frac{S_n^\lambda f(x)}{(n+1)^{\lambda +1}} \right|. \]

The following lemmas are finitary adaptations of the results in \cite{NS}; we emphasize that the formal nature of the arguments in \cite{NS} allows them to be applied to any Ces{\`a}ro means of sequence of Markov operators operators that are $L^1$ and $L^\infty$ contractions.

\medskip

\begin{lemma}[\cite{NS}, Proof of Lemma 4, pp. 144-145]\label{1}
For $\alpha > 0, \beta \in \RR$, there exist positive constants $C_\alpha$ so that
\[ S_*^{\alpha+i\beta} f \leq C_\alpha e^{2 \beta^2} S_*^0|f| \]
holds pointwise.
\end{lemma}

\medskip

\begin{lemma}[\cite{NS}, Proof of Lemma 5, pp. 145-146]\label{2}
For each nonnegative integer $t$ and each real $\beta$, there exist positive constants $C_t$ so that
\[ S^{-t+i \beta}_*f \leq C_t e^{3 \beta^2} \left( S_*^{-t-1}f + S_*^{-t}f + \dots + S_*^{-1} f \right)\]
holds pointwise.
\end{lemma}

\medskip

\begin{lemma}[\cite{NS}, Proof of Lemma 5, p. 147]\label{3}
Define
\[ R_t f(x)^2 := \sum_{0 \leq k \leq c_mN} (k+1)^{2t-1}|S_k^{-t-1}f(x)|^2 \]
for any positive integer $t$. Then there exists a positive constant $c_{-t}$ so that
\[ S_*^{-t} f \leq c_{-t} R_t f  + 2S_*^{1-t}f\]
holds pointwise.
\end{lemma}

\begin{proposition}\label{main}
Let
\[R_t f(x)^2 := \sum_{0 \leq k \leq c_m N} (k+1)^{2t-1}|S_k^{-t-1}f(x)|^2.\]
Then there exist constants $C_{t,m}$ independent of $N$ so that
\[ \|R_t f \|_{L^2(\Z_{m+1}^N)} \leq C_{t,m} \| f\|_{L^2(\Z_{m+1}^N)}\]
for all $N$.
\end{proposition}

Before the proof of Proposition \ref{main}, we show that it implies dimension independent $L^p$ bounds on $M^L$.

\begin{proof}[Theorem \ref{localbound}, Assuming Proposition \ref{main}]\label{localboundproofassumingmain}

First we note that $S_*^0$ is the smooth local spherical maximal operator $M^L_S$ from Proposition \ref{smoothweaktype} while $S_*^{-1}$ is the local spherical maximal operator $M^L$ so our goal is to establish dimension independent $L^p$ bounds on $S_*^{-1}$.

By Corollary \ref{smoothlpbounds}, we know that there exist constants $\{A_{p,m}\}$, $p > 1$, so that for each $N$,
\[ \big\| S_*^0 |f| \big\|_{L^p(\Z_{m+1}^N)} \leq A_{p,m} \|f\|_{L^p(\Z_{m+1}^N)},\]
where the operators $\{S_*^0\}$ are $N$-dependent, but the bounds are not.

By Lemma \ref{1}, for each $\alpha > 0, \beta \in \RR$, we therefore have the bound
\[ \| S_*^{\alpha+i\beta} f \|_{L^p(\Z_{m+1}^N)} \leq C_\alpha e^{2 \beta^2} A_{p,m} \| f \|_{L^p(\Z_{m+1}^N)}\]
independent of $N$.

By Proposition \ref{main}, Lemma \ref{3}, and induction on $t$, we see that there exist constants $\{B_{2,m}^t\}, t \geq 1$ so that for all $N$,
\[ \|S_*^{-t} f\|_{L^2({\Z_{m+1}}^N)} \leq B_{2,m}^t \|f\|_{L^2({\Z_{m+1}}^N)}.\]
By Lemma \ref{2}, this means that for all $N$, there exist constants $\{D_{2,m}^t \}$ so that
\[ \|S_*^{-t+i\beta} f\|_{L^2({\Z_{m+1}}^N)} \leq e^{3\beta^2} D_{2,m}^t \|f\|_{L^2({\Z_{m+1}}^N)} \]
for all $N$.

The theorem then follows by linearizing the $S_*^{-1}$-supremum and applying the Stein interpolation theorem as in the conclusion of the proof of \cite[Theorem 2, pp. 150-151]{NS}.
\end{proof}

It remains only to prove Proposition \ref{main}, which we accomplish in the following subsection.

\subsection{Proof of Proposition \ref{main}}
\begin{proof}
We proceed by truncating $R_t f(x)^2$ after $t$ summands and bound the tail later. Each individual operator $S_k^{-t-1}$ is bounded in $L^2(\Z_{m+1}^N)$ with a bound dependent on $k$ and $t$. Thus, letting $c_t := \max_{k < t} \|S_k^{-t-1}\|_2$,
\[\aligned
\sum_{x \in \Z_{m+1}^N} \sum_{k = 0}^{t-1} (k+1)^{2t-1} |S_k^{-t-1}f(x)|^2 &= \sum_{0 \leq k < t} (k+1)^{2t-1} \|S_k^{-t-1}f(x)\|_2^2\\
&\leq c_t^2 \sum_{0 \leq k < t} (k+1)^{2t-1} \|f\|_2^2\\
&\lesssim_t \|f\|_2^2
\endaligned\]

Now we move on to establish the desired bound for the tail, namely
\[
\sum_{k=t}^{c_m N} (k+1)^{2t-1} |S_k^{-t-1}f(x)|^2 = \sum_{k = t}^{c_m N} (k+1)^{2t-1} |\triangle^{t} P^k f(x)|^2 \lesssim_t \|f\|_2.
\]

By Plancherel, it is enough to show that there exists a constant, $C'_{t,m}$, independent of $r$ and $N$, so that for all $r$
\[ \sum_{k=t}^{c_mN} (k+1)^{2t-1} \ee|\F \triangle^{t} P^k\rr|^2(r)
\leq C'_{t,m} \]

or equivalently

\begin{align}\label{krawtchoukdifferencesum}
\sum_{k=t}^{c_mN} (k+1)^{2t-1} \ee|\partial^{t} \kappa_r^N(k)\rr|^2 \leq C'_{t,m}.
\end{align}

If $r<t$, each summand is $0$ so without loss of generality $r \geq t$. Ignoring a finite set of cases for fixed $t$, we can assume that $N > 2t$. Vital to the proof is the difference relation
\[ \partial^{t} \kappa_r^N(k)= \ee(-\frac{1}{c_m}\rr)^t \frac{\binom{N-t}{r-t}}{\binom{N}{r}} \kappa^{N-t}_{r-t}(k-t) \]
from Corollary \ref{krawtchoukhigherderivatives} and the upper bound
\[ |\kappa_{r-t}^{N-t}(k-t)| \leq e^{-d \frac{r-t}{N-t} (k-t)}\]
from Proposition \ref{tech}.

We first handle of the boundary case $r=t$, in which
\[ \kappa_{r-t}^{N-t}(k-t) = \kappa_{0}^{N-t}(k-t) = 1.\]
In this instance, we estimate
\[ \aligned
\sum_{k=0}^{c_mN} (k+1)^{2t-1} |\partial^{t} \kappa_r^N(k)|^2 &\leq
\sum_{k=1}^{N} k^{2t-1} \left( \frac{(c_m)^{-t}}{\binom{N}{t}} \right)^2 \\
&\lesssim \left( \frac{(N/c_m)^t}{\binom{N}{t}} \right)^2 \\
&\lesssim_t 1, \endaligned\]
simply bounding $\binom{N}{t}$ from below by $(N-t)^t/t^t \approx_t N^t$ because $N > 2t$. Henceforth, we may assume $r>t$. 

Seeking the bound \eqref{krawtchoukdifferencesum}, we estimate
\[ \aligned
\sum_{k=t}^{c_mN} (k+1)^{2t-1} |\partial^{t} \kappa_r^N(k)|^2 &\leq
\sum_{k=t}^{\infty} (k+1)^{2t-1} |\partial^{t} \kappa_r^N(k)|^2 \\
&\lesssim_t \sum_{k=t}^{\infty} (k+1)^{2t-1} \left| \frac{\binom{N-t}{r-t}}{\binom{N}{r}} e^{-d \frac{r-t}{N-t} (k-t)} \right|^2 \\
&= \left( \frac{\binom{N-t}{r-t}}{\binom{N}{r}} \right)^2 \sum_{k=t}^{\infty} (k+1)^{2t-1} e^{-2d \frac{r-t}{N-t} (k-t)} \\
&= \left( \frac{\binom{N-t}{r-t}}{\binom{N}{r}} \right)^2 \sum_{k=0}^{\infty} \big(k+(t+1)\big)^{2t-1} e^{-2d \frac{r-t}{N-t} k}. \endaligned\]

We record the following easy lemma concerning infinite series:
\begin{lemma}\label{infiniteseries}
For any positive integer $n$, there exists a constant $A_n$ such that for all $|s| < 1$,
\[
\sum_{k=0}^{\infty} k^n s^k \leq \frac{A_n}{(1-s)^{n+1}}
\]
\end{lemma}

\begin{proof}
Define the operator
\[ Lf(s):= s \frac{df}{ds}(s),\]
and note that
\[ \sum_{k=0}^{\infty} k^n s^k = L^n \frac{1}{1-s}\]
Induction on $n$ shows that the right side of this equation can be expressed as $\frac{s^n + p_n(s)}{(1-s)^{n+1}},$ where $p_n(s) := \sum_{j <n} a_j^n s^j$ is a polynomial of degree $n-1$.

In particular, for $s < 1$, we may bound
\[ \left|\frac{s^n + p_n(s)}{(1-s)^{n+1}} \right| \leq \frac{A_n}{(1-s)^{n+1}},\]
where we let $A_n:= 1+\sum_{j<n} |a_j^n|$.
\end{proof}

Now, following the lead (and notation) of \cite[\S 4]{HKS}, we set
\[ \alpha = \alpha(r):= 2d \frac{r-t}{N-t};\]
possibly after reducing $d$, we may assume that $d< \frac{1}{2}$,
so that
\[ |\alpha(r)| \leq 2d \frac{r-t}{N-t} \leq 2d < 1\]
for all $r$.

In the following estimate we use the notation $A_n$ from Lemma \ref{infiniteseries}.
\[ \aligned
\sum_{k=0}^{\infty} \big(k+(t+1)\big)^{2t-1} e^{-2d \frac{r-t}{N-t} k} &=
\sum_{k=0}^{\infty} \big(k+(t+1)\big)^{2t-1} e^{-\alpha k} \\
&= \sum_{k=0}^{\infty} \left( \sum_{n=0}^{2t-1} \binom{2t-1}{n} k^n (t+1)^{2t-1-n} \right) e^{-\alpha k} \\
&\leq \sum_{k=0}^{\infty} \left( \sum_{n=0}^{2t-1} (2t)^{2t} k^n (t+1)^{2t} \right) e^{-\alpha k} \\
&\lesssim_t \sum_{n=0}^{2t-1} \left( \sum_{k=0}^{\infty} k^n e^{-\alpha k} \right) \\
&\lesssim_t \sum_{k=0}^{\infty} k^{2t-1} e^{-\alpha k} \\
&\leq \frac{A_{2t-1}}{(1-e^{-\alpha})^{2t}}\\
&\lesssim_t \alpha^{-2t}, \endaligned \]
where we used the mean value theorem in passing last line.

The upshot is that we may bound
\[ \sum_{k=0}^{\infty} (k+(t+1))^{2t-1} e^{-2 d \frac{r-t}{N-t} k} \lesssim_t \left( \frac{N-t}{r-t} \right)^{2t},\]
so that we have
\[
\sum_{k=0}^{c_mN} (k+1)^{2t-1} |\partial^{t} \kappa_r^N(k)|^2 \lesssim_t
\left[\left( \frac{\binom{N-t}{r-t}}{\binom{N}{r}} \right) \left( \frac{N-t}{r-t} \right)^{t} \right]^2.\]

Of course uniformly in $0 \leq j \leq t$, we have the equivalence $N-j \approx_t N$ and $r-j \approx_t r$ (recall $r \geq t+1$) so direct computation shows
\[
\frac{\binom{N-t}{r-t}}{\binom{N}{r}} \approx_t  \ee(\frac{r}{N}\rr)^t \text{ and } \ee(\frac{N-t}{r-t}\rr)^t \approx_t \ee(\frac{N}{r}\rr)^t,
\]
thus proving the bound.\end{proof}

\section{Distant Spheres}\label{distant}

The strategy for bounding maximal averages over distant spheres is to bound (up to a constant) the smooth distant spherical maximal operator
\[
M^D_S f = \sup_{D \leq \frac{N}{m+1}} \bigg|\frac{1}{D+1} \sum_{d \leq D} \sigma_{N-d} *f\bigg|
\]
by the maximal operator given by precomposing the smooth local spherical maximal operator by $P^N$, the outermost spherical average. Explicitly, this operator is
\[
\sup_{L \leq c_m N} \ee|\frac{1}{L+1} \sum_{l \leq L} \big(\sigma_l * \sigma_N * f\big)(x)\rr|.
\]

The latter operator inherits the dimension independent $L^p$ bounds on $M^L_S$ from Corolloary \ref{smoothlpbounds} simply becasue $P^N$ is an $L^p$ contraction for all $1 \leq p \leq \infty$. Once $L^p$ bounds are established for $M^D_S$, the  arguments from $\S \ref{comparison}$ work similarly to bound $M^D$.

\begin{lemma}\label{distantiid}
For any $k \leq c_m N$, 
\[
\sigma_k * \sigma_N = \sum_{d \leq k} b_k(d) \sigma_{N-d}
\]
where $b_k(d)$ is the probability mass of a sum of $k$ i.i.d. copies of a random variable
\[
X :=
\begin{cases}
0 \text{ with probability } \frac{m-1}{m} \\
1 \text{ with probability } \frac{1}{m}.
\end{cases}
\]
\end{lemma}

\begin{proof}

Notice that $\sigma_k * \sigma_N(x)$ is a nonnegative function with integral $1$, supported on $\{x \in \Z_{m+1}^N: |x| \geq N-k\}$. First we show that this function is radial by fixing $x$ such that $N-k \leq |x| \leq N$ and observing that the number of pairs $(y,z) \in \Bbb S_k \times \Bbb S_N$ such that $x = y+z$ depends only on $|x|$.

To see this, we partition the pairs into sets $S_j$ containing all those $(y,z)$ such that exactly $j$ of the nonzero components of $y$ (note that there are $k$ such components in total) have indices $i$ such that $x_i = 0$. A counting argument shows that
\[
|S_j| = \binom{N-|x|}{j} m^j \binom{|x|}{k-j} (m-1)^j.
\]
Summing up $|S_j|$ from $j=0$ to the lesser of $N-|x|$ and $|x|-k$ proves radiality. Thus we may write
\[
\sigma_k * \sigma_N(x) = \sum_{d \leq k} b_k(d) \sigma_{N-d}
\]

with $b(0) + \dots + b(k) = 1$ and $b(0), \dots , b(k) > 0$.

Another counting argument shows that for any fixed $0 \leq d \leq k$ and $z \in \Bbb S_N$,
\[
\ee|\ee\{y \in \Bbb S_k: |y+z|=N-d\rr\}\rr| = \binom{N}{k} \binom{k}{d} (m-1)^{k-d}.
\]
Therefore
\begin{align}\label{blexpression}
b_k(d) &= \ee\langle \sigma_k * \sigma_N, 1_{\Bbb S_{N-d}}\rr\rangle\nonumber\\
&= \frac{1}{|\Bbb S_k| |\Bbb S_N|} \ee|\ee\{(y,z) \in \Bbb S_k \times \Bbb S_N: |y+z|=N-d\rr\}\rr|\nonumber\\
&= \frac{1}{|\Bbb S_k|}  \binom{N}{k} \binom{k}{d} (m-1)^{k-d}\nonumber\\
&= m^{-k} \binom{k}{d} (m-1)^{k-d}.
\end{align} 

Finally we define a discrete random variable
\[
X :=
\begin{cases}
0 \text{ with probability } \frac{m-1}{m} \\
1 \text{ with probability } \frac{1}{m}
\end{cases}
\]
and directly compute that \eqref{blexpression} is exactly the probability that $k$ i.i.d. copies of $X$ sum to $d$.
\end{proof}

\begin{remark}\label{probabilisticexplanation}
The intuition for this result is the the convolution of $\sigma_k$ and $\sigma_N$ can be thought of as the following random process:
\begin{enumerate}
\item Pick an element of $\Bbb S_N$ uniformly at random.
\item Pick $k$ components to change uniformly at random (among $k$-subsets of $[N]$).
\item Independently choose one of the remaining $m$ values in $\Bbb Z_{m+1}$ for each of those $k$ components.
\end{enumerate}
The symmetries of the first 2 steps above easily imply that the probability mass on $\Bbb Z_{m+1}^N$ is radial. Moreover, the length of the output is independent of the first 2 steps so it is simply determined by the outcome of the final step; a $k$-fold i.i.d. process with a $\frac{1}{m}$ probability of decreasing the length by $1$ and a $\frac{m-1}{m}$ probability of preserving the length.
\end{remark}

\begin{lemma}\label{blbound}
Let $0 \leq d \leq N/m$. Then for any $j$ within $\sqrt d$ of $md$ , $b_j(d) \gtrsim d^{-1/2}$.
\end{lemma}

\begin{proof}
This lemma can be thought of as a pointwise application of the central limit theorem. Indeed we start by noting that by the (classical) central limit theorem, the expressions
\begin{align}\label{stddeviationsums}
\sum_{i = j/m - 2\sqrt d}^{j/m- \sqrt d} b_j(i) \text{ and } \sum_{i = j/m + \sqrt d}^{j/m + 2 \sqrt d} b_j(i)
\end{align}
converge to positive numbers as $j \to \infty$ (which is equivalent to $d \to \infty$). To see this, note that probability mass (on the variable $a$)
\[
b_j\left(a\sqrt j +\frac{j}{m}\right)
\]
converges weakly to a fixed Gaussian. Moreover, $\sqrt j \approx \sqrt d$ so both expressions in \eqref{stddeviationsums} converge to integrals of this Gaussian over fixed intervals. In particular, this implies that there exist
\[
\lambda \in [j/m - 2\sqrt d, j/m - \sqrt d], \rho \in [j/m + \sqrt d, j/m + 2 \sqrt d]
\]
such that $b_j(\lambda), b_j(\rho) \gtrsim d^{-1/2}$.

Recall from Lemma \ref{blexpression} that $b_j(i) = m^{-j} \binom{j}{i} (m-1)^{j-i}$. In the interest of proving a concavity property of $b_j$, we observe that the ratio of successive summands is
\[
R_j(i) := \frac{b_j(i+1)}{b_j(i)} = \frac{j-i}{(m-1)(i+1)}.
\]
Notice that $R_j$ decreases from $i=\lambda$ to $i=\rho$ and that $d \in [\lambda, \rho]$ (this can be computed directly from the definitions of $j$ and $d$). Therefore, if $b_j(d) \leq b_j(\lambda)$, $R_j(d) \leq 1$. Moreover, if  $R_j(d) \leq 1$, then $b_j(d) \geq b_j(\rho)$.

Thus, at least one of the inequalities $b_j(d) \geq b_j(\lambda)$ and $b_j(d) \geq b_j(\rho)$ must hold. Either way this shows $b_j(d) \geq d^{-1/2}$.

\end{proof}

\begin{proposition}
For any nonnegative function $f$ we have the pointwise inequality
\[
M^D_S f \lesssim M^L_S \big(\sigma_N * f\big)(x).
\]
in particular $M^D_S$ is weak-type $1-1$, independent of dimension. Again, by interpolation this implies dimension independent $L^p$ bounds for all $1<p\leq\infty$.
\end{proposition}

\begin{proof}
Because the operators in question are suprema over positive convolution operators, we seek pointwise bounds on the convolution kernels. Moreover, it suffices to show that any $0 \leq L \leq c_mN$,
\begin{align}\label{distantspheretrasfer}
\sum_{d \leq L/m} \sigma_{N-d} \lesssim \sum_{l \leq L} \sigma_l * \sigma_N.
\end{align}
This can be seen simply by dividing the left and right sides by $L/m + 1$ and $L+1$ respectively (as these values are equivalent up to a constant) and taking a supremum over $L$. Applying Lemmas \ref{distantiid} and \ref{blbound}, the right side of \eqref{distantspheretrasfer} can be reformulated:
\begin{align*}
\sum_{l \leq L} \sigma_l * \sigma_N &= \sum_{l=0}^L \sum_{j=0}^l b_k(j) \sigma_{N-j}\\
&= \sum_{j=0}^L \sum_{l=j}^L b_l(j) \sigma_{N-l}\\
&\geq \sum_{d = 0}^{L/m} \sum_{l=dm-\sqrt d}^{dm} b_l(d) \sigma_{N-d}\\
&\gtrsim \sum_{l \leq L/m} d^{-1/2} d^{1/2} \sigma_{N-d}\\
&\geq \sum_{d \leq L/m} \sigma_{N-d}.
\end{align*}

\end{proof}

Because the interpolation techniques used in \S \ref{comparison} apply to any Ces{\`a}ro means for a sequence of Markov operators that are $L^1$ and $L^\infty$ contractions, much of the argument caries over with the modification that the opertor $P^k$ is replaced by
\[
Q^k f(x) := f * \sigma_{N-k}(x),
\]
the operator $S_n^\lambda$ is replaced by
\[
T_n^\lambda f(x) := \sum_{k \leq n} A^\lambda_{n-k} Q^k f(x), \lambda \in \C,
\]
and the operator $S_*^\lambda$ is replaced by
\[
T_*^\lambda f(x) := \max_{0 \leq n \leq \frac{N}{m+1}} \ee|\frac{T_n^\lambda f(x)}{(n+1)^{\lambda+1}}\rr|,
\]
and the operator $R_t$ is replaced by
\[
\ppt f(x)^2 := \sum_{0 \leq k \leq \frac{N}{m+1}} (k+1)^{2t-1}|T_k^{-t-1} f(x)|^2.
\]

All other definitions from \S \ref{comparison} are translated over analogously (of course the local operators will be replaced by distant operators). Note also that, following the computations of Lemma \ref{krawtchoukderivatives} and Corollary \ref{krawtchoukhigherderivatives},
\begin{align}\label{krawtchoukforwardderivative}
\big|(\F \triangle^t Q^k)(r)\big| =\ee(\frac{1}{c_m}\rr)^t\frac{\binom{N-t}{r-t}}{\binom{N}{r}} |\kappa^{N-t}_{r-t}(N-k)|.
\end{align}

Carrying over the proof of Theorem \ref{localbound} in the natural way, we can establish (modulo an analog to Proposition \ref{main}) the distant spherical bound, from which the main result Theorem \ref{MAIN} follows:

\begin{theorem}\label{distantbound}
The distant spherical maximal operator $M^D$ given by
\[
M^D f(x) = \sup_{k \leq \frac{N}{m+1}} |\sigma_{N-k} * f(x)|
\]
satisfies $L^p$ bounds for all $p > 1$ dependent only on $p$ and $m$.
\end{theorem}

Thus, the only remaining element in this section is the distant sphere square function bound.

\begin{proposition}\label{distantsquarefunction}
With
\[ \aligned
\ppt f(x)^2 &= \sum_{0 \leq k \leq \frac{N}{m+1}} (k+1)^{2t-1}|T_k^{-t-1}f(x)|^2 \\
&= \sum_{0 \leq k \leq \frac{N}{m+1}} (k+1)^{2t-1} |\triangle^{t} Q^k f(x)|^2, \endaligned \]
there exist constants $C_{t,m}$ independent of $N$ so that
\[ \|\ppt f \|_{L^2(\Z_{m+1}^N)} \leq C_{t,m} \| f\|_{L^2(\Z_{m+1}^N)}\]
for all $N$.
\end{proposition}

\begin{proof}
Much of the proof of Proposition \ref{main} carries over. In fact, the fact that all spheres appearing in $\ppt$ have radii on the order of $N$ makes the bound simpler.

For any $r \geq t$, we bound
\begin{align*}
\sum_{k=0}^{\frac{N}{m+1}} (k+1)^{2t-1} |\F \Delta^t Q^k|^2(r) &= \sum_{k=0}^{\frac{N}{m+1}} (k+1)^{2t-1} \ee(\frac{1}{c_m}\rr)^{2t} \ee(\frac{\binom{N-t}{r-t}}{\binom{N}{r}} \ee|\kappa^{N-t}_{r-t}(N-k)\rr|\rr)^2\\
&\lesssim_t \sum_{k=0}^{\frac{N}{m+1}} (k+1)^{2t-1} \frac{\binom{N-t}{r-t}^2}{\binom{N}{r}^2} \exp\ee(-2d\frac{(r-t)(N-k)}{N-t}\rr)\\
&\leq \sum_{k=0}^{\frac{N}{m+1}} (k+1)^{2t-1} \frac{\binom{N-t}{r-t}^2}{\binom{N}{r}^2}e^{-d(r-t)}\\
&\lesssim_t N^{2t} \ee(\frac{r}{N}\rr)^{2t} e^{-dr}
\end{align*}
where we used the fact that $k \leq \frac{N}{m+1}$ to pass to the second-to-last line and the estimates at the end of the proof of Proposition \ref{main} to pass to the last line. Because $e^{-dr} \lesssim_t r^{-2t}$, this proves the desired inequality 
\[
\sum_{k=0}^{\frac{N}{m+1}} (k+1)^{2t-1} |\F \Delta^t Q^k|^2(r) \lesssim_t 1.
\]
\end{proof}

\section{Proof of Proposition \ref{tech}}\label{krawtchouk}
First we introduce the notation
\[
a_j =\binom{N}{k}^{-1} \binom{r}{j} \binom{N-r}{k-j} m^{-j}
\]
for the magnitude of the $j$th summand in the full expression for $\kappa^N_k(r)$, which we recall is given by 
\begin{align}\label{krawtchoukdefinition}
\kappa^N_k(r) = \sum_{j=\max(0,r+k-N)}^{\min(r,k)} (-1)^j \binom{N}{k}^{-1} \binom{r}{j} \binom{N-r}{k-j} m^{-j}.
\end{align}

We restate the proposition for the reader's convenience:

\begin{prop}[restatement]
There exists a constant $d$ (dependent only on $m$) such that for all $r,k,N$ we have
\[ |\kappa_k^N(r)| \leq e^{-d (rk/N)}.\]
\end{prop}

By the symmetry of the Krawtchouk polynomials in $r$ and $k$, without loss of generality $r \leq k$ so the sum \eqref{krawtchoukdefinition} will terminate at $r$. The thrust of the proof is to show that the largest summand magnitude in \eqref{krawtchoukdefinition} decays exponentially in $rk/N$ so Lemma \ref{dominantterm} below will prove the proposition.

For the remainder of the section we define
\[
\ell := \max(0, r+k-N)
\]
to be the lowest index of summation. Also we define $n$ to be the lowest index in the region of summation, i.e. $[\ell, r] \cap \Z$, such that $a_n$ is a maximal summand magnitude. In other words, $n \in \Z$ is minimal subject to the constraints that $\ell \leq n \leq r$ and $a_j \leq a_n$ for all $j \in \Z$ in that range.

\begin{lemma}\label{dominantterm}
Each Krawtchouk polynomial is dominated by its maximal summand magnitude. More concretely, $\left|\kappa^N_k(r)\right| \leq a_n$.
\end{lemma}

\begin{proof}
We begin by noting that the ratio $a_{j+1}/a_j$ is given by
\[
R(j) := \frac{(r-j)(k-j)}{m(j+1)(j+N-r-k+1)}.
\]
We view $R$ as a function on the real interval $(\ell - 1, r]$ rather than restricting it to the integers. Its key properties for this lemma are
\begin{enumerate}[(i)]
\item $R(j) \geq 0$,
\item $R(j)$ is continuously (strictly) decreasing,
\item $R(j)$ approaches $+\infty$ as $j$ approaches $\ell - 1$, and
\item $R(r) = 0$.
\end{enumerate}
Property (i) above follows from the fact that all factors in $R(j)$ are nonnegative. Property (ii) is a result of the factors in the numerator diminishing in magnitude and the factors in the denominator growing. Property (iii) follows from property (i) and the fact that $R$ has a pole at $\ell - 1$ while property (iv) is trivial.

By the intermediate value theorem, properties (ii), (iii), and (iv) imply that there exists some $J \in (\ell-1, r)$ such that $R(J)=1$. Applying property (ii), we see that for all integers $j$ in the region of summation,
\begin{align}\label{krawtchoukmonotonicity}
\begin{split}
&j \leq J \implies R(j) \geq 1 \implies a_{j+1} \geq a_j\\
&j \geq J \implies R(j) \leq 1 \implies a_{j+1} \leq a_j.
\end{split}
\end{align}
In particular, this means that $a_{\lceil J \rceil}$ is a maximal summand magnitude. Note that because $R(j)$ is strictly decreasing, $R(\lceil J \rceil - 1) > 1$ so $a_{\lceil J \rceil} > a_{\lceil J \rceil - 1}$. Thus $\lceil J \rceil$ must minimal among indices of maximal summand magnitudes, i.e. $n = \lceil J \rceil$.

Finally, we can bound $\left|\kappa^N_k(r)\right|$ by splitting it into two monotonic alternating sums, namely
\[
\kappa^N_k(r) = \left(\sum_{j=0}^n (-1)^j a_j\right) + \left(\sum_{j=n+1}^r (-1)^j a_j\right)
\]
where the monotinicity is a direct consequence of \eqref{krawtchoukmonotonicity}. Note that the second sum above may be empty, but we can ignore this by defining $a_{r+1}$ to be $0$.

Because they are monotonic and alternating, the sums are bounded between $0$ and their respective largest magnitude summands, namely $\pm a_n$ and $\mp a_{n+1}$. Because these bounds have opposite signs, we can bound $\left|\kappa^N_k(r)\right|$ by the maximum of their magnitudes, namely $a_n$.
\end{proof}

To bound $a_n$ we first bound $n$ from below. This technical lemma is largely comprised of algebraic and calculus manipulations.

\begin{lemma}\label{peaklemma}
If $n > 0$ and $rk \geq 2Nm$ then $n \gtrsim rk/N$.
\end{lemma}

For the sake of clarity we point out that the hypothesis $rk \geq 2Nm$ proves $n > 0$ a posteriori, however it is more efficient to handle the $n = 0$ case separately.

\begin{proof}
We recall from Lemma \ref{dominantterm} that the ratio $a_{j+1}/a_j$ is given by
\[
R(j) := \frac{(r-j)(k-j)}{m(j+1)(j+1+N-r-k)}.
\]

To solve the equation $R(j) = 1$, we apply the quadratic formula to the quadratic
\[ \aligned
&\big[m(j+1)(j+1+N-r-k)\big]-\big[(r-j)(k-j)\big]\\
&=(m-1)j^2 + [2m+Nm-(m-1)(r+k)]j + [m+Nm-rk-km-rk].
\endaligned \]

This reveals that $R$ can equal $1$ only at the values
\[
j_{\pm} := C \pm \sqrt{C^2+ A}
\]

Where 
\begin{align}
A &:= -\frac{4(m-1)(m+Nm-rm-km-rk)}{4(m-1)^2} = \bigg(\frac{rk-Nm}{m-1} + \frac{rm+km-m}{m-1}\bigg)\label{defA}\\
C &:= -\frac{2m+Nm-(m-1)(r+k)}{2(m-1)} = \bigg(\frac{r+k}{2} - \frac{m}{m-1} - \frac{Nm}{2(m-1)}\bigg)\label{defC}.
\end{align}

We will show
\begin{enumerate}[(I)]
\item $A > 0$,
\item $A \gtrsim rk$, and
\item $\sqrt{C^2 + A} - |C| \gtrsim rk/N$.
\end{enumerate}

Item (I) above implies that $j_- < 0 < j_+$. We saw in the proof of Lemma \ref{dominantterm} there exists $J \in (\ell - 1,r)$ such that $R(J) = 1$ and $n = \lceil J \rceil$. It follows that $J = j_{\pm}$ and, by the assumption $n > 0$, that $J>0$. Therefore $J = j_+$ simply by default.

Item (II) is the key element in the proof of item (III). Item (III) shows that
\[
n \geq J \gtrsim rk/N
\]
simply because, regardless of the sign of $C$,
\[
J = C + \sqrt{ C^2 + A} \geq \sqrt{ C^2 + A} - |C| \gtrsim rk/N.
\]
Therefore all that remains in the lemma is to justify (I), (II), and (III).

\bigskip
\emph{Justification of (I) and (II):}

In light of fact that $r$ and $k$ are positive integers, simple arithmetic shows that
\[
\frac{rm+km-m}{m-1} > 0.
\]

and, because $rk \geq 2Nm$,
\[
\frac{rk-Nm}{m-1} \geq \frac{1}{2(m-1)} rk.
\]
Adding these two inequalities, the last expression in \eqref{defA} shows $A >0$ and $A \gtrsim rk$.

\bigskip
\emph{Justification of (III):}

We split into two cases.

\begin{enumerate}[\text{Case }1:]
\item If $A > 3C^2$, then
\[ \sqrt{C^2 + A} - |C| \geq A^{1/2} - (A/3)^{1/2} \gtrsim A^{1/2} \]

We know that $A \gtrsim rk$ and $(rk)^{1/2} \leq N$ by item (II) and the bound $r, k \leq N$ respectively. It follows that
\[ A^{1/2} \gtrsim (rk)^{1/2} = \frac{rk}{(rk)^{1/2}} \geq \frac{rk}{N}.\]

\item If $A \leq 3C^2$, then we apply the mean value theorem to observe that
\begin{align*}
\sqrt{C^2+A}-|C| &\geq \ee(\inf_{x \in [C^2, C^2 + A]} \frac{1}{2x^{1/2}}\rr) A\\
&\geq \frac{A}{2(4C^2)^{1/2}}\\
&\gtrsim \frac{rk}{N}.
\end{align*}

The final inequality follows from the bounds $A \gtrsim rk$ and $|C| \lesssim N$. The former is again item (II) and the latter comes from the fact that each term in the last expression of \eqref{defC} is bounded in magnitude by $2$ or $N$.
\end{enumerate}

Thus, regardless of $A$, $\sqrt{C^2 + A} - |C| \gtrsim rk/N$.
\end{proof}

From here Proposition \ref{tech} is fairly straightforward.

\begin{proof}[Proof of Proposition \ref{tech}]

First we use the combinatorial observation
\[ \aligned
\binom{N}{k} &= \bigg|\bigg\{S \subset [N]: |S| = k\bigg\}\bigg|\\
&\geq \bigg|\bigg\{S \subset [N]: |S| = k, \big|S \cap [r]\big| = j\bigg\}\bigg| = \binom{r}{j} \binom{N-r}{k-j}
\endaligned \]
to justify the inequality
\begin{align}\label{cubecase}
a_j = \binom{N}{k}^{-1} \binom{r}{j} \binom{N-r}{k-j} m^{-j} \leq m^{-j}.
\end{align}
for all $j$ in the region of summation.

This bound is useful because in order to prove the proposition, it is sufficient to prove $a_n \leq e^{-d(rk/N)}$ by Lemma \ref{dominantterm}. To this end, we split into three cases.
\begin{enumerate}[\text{Case } 1:]
\item The hypotheses of Lemma \ref{peaklemma} hold, i.e. $n > 0$ and $rk \geq 2mN$. Then there is an index $n$ and a (small) constant $\epsilon >0$ such that
\[
n \geq \epsilon \frac{rk}{N}.
\]
Letting $d := \epsilon\ln m > 0$, this shows $a_n \leq e^{-d(rk/N)}$ by \eqref{cubecase}.

\item $n>0$ and $rk < 2mN$. Because $m \geq 2$ and $n \geq 1$ by assumption, \eqref{cubecase} provides the inequality $a_n \leq 1/2$. Moreover, the assumption $rk < 2mN$ implies that
\[
e^{-2m} \leq e^{-rk/N}.
\]
Then we simply decrease $d$ to a small enough (positive) number that $1/2 \leq e^{-2md}$, to achieve the desired bound
\[
a_n \leq 1/2 \leq e^{-2md} \leq e^{-d(rk/N)}.
\]

\item $n=0$. We assume $r > 0$ because otherwise the entire proposition is trivial. Also, because
\[
\max(0, r+k-N) \leq n = 0,
\]
we know that $r+k \leq N$ so the factors below are all well defined. Then we bound as follows:
\[ \aligned
a_0 &= \binom{N}{k}^{-1} \binom{N-r}{k}\\
&= \prod_{j=0}^{k-1} \frac{N-r-j}{N-j}\\
&\leq \left(\frac{N-r}{N}\right)^k\\
&= \left[\left(1 - \frac{r}{N} \right)^{\frac{N}{r}}\right]^{rk/N}\\
&\leq e^{-rk/N}
\endaligned \]
Because we are free to assume $d \leq 1$, this completes the proof of proposition \ref{tech}.
\end{enumerate}\end{proof}

\end{document}